\documentclass{amsart}
\usepackage{amsfonts}
\usepackage{graphics,amsmath,amsthm,amssymb}
\usepackage{color}
\newtheorem{theorem}{Theorem}[section]

\newtheorem{corollary}[theorem]{Corollary}

\newtheorem{remark}[theorem]{Remark}
\newtheorem{lemma}[theorem]{Lemma}

\begin{document}
\title[On the 3-rank  
of the class group  of quadratic fields]{On the 3-rank  
of the class group  of quadratic fields}
 \author[S.-C. Chen]{Shi-Chao Chen}

 \address{ Department of Mathematics
 and Statistics Sciences, Henan University, Kaifeng, 475004, China}
 \email{schen@henu.edu.cn}

 \author[C.-C. Wu]{Chuan-Chuan Wu}
 \address{ Department of Mathematics
 and Statistics Sciences, Henan University, Kaifeng, 475004, China}
 \email{wcc@henu.edu.cn}



\begin{abstract}
Let $n\ge1$,   $r\ge0$ and $s\ge0$ be integers satisfying $4+r+3 s\le3^{n+1}$. Given linear polynomials
$f_{i}(x)=m_{i} x+n_{i}$ for $1 \le i \le r+s$, where  the   coefficients $m_{i} , n_{i}$ are positive integers satisfying certain conditions, we prove that  there exist infinitely many fundamental discriminants $D>0$  such that the 3-rank of the class group of each quadratic fields
$\mathbb{Q}(\sqrt{f_1(D)}), \ldots, \mathbb{Q}(\sqrt{f_r(D)})$ and  $\mathbb{Q}(\sqrt{-f_{r+1}(D)}), \ldots, \mathbb{Q}(\sqrt{-f_{r+s}(D)})$ is  simultaneously less than $n$.
Moreover, for any positive integer $k$, there exist positive integers $a, d$ such that the 3-rank of the class group of   each quadratic fields
$\mathbb{Q}(\sqrt{a+g_1(d)}), \ldots,\mathbb{Q}(\sqrt{a+g_k(d)})$ is simultaneously less than $n$  for polynomials $g_1(x), g_2(x), \ldots, g_k(x)$  that take integer values at the integers and have no constant terms.
\end{abstract}

\maketitle

\section{Introduction}
Let \(K = \mathbb{Q}(\sqrt{D})\) be the quadratic field corresponding to a fundamental discriminant \(D\), and denote by \(\mathrm{Cl}_{D}\) its class group. For a prime \(p\), the \(p\)-rank of a finite abelian group \(G\) is defined as \(\mathrm{r}_{p}(G) = \dim_{\mathbb{F}_{p}}(G / G^{p})\); for brevity we write \(r_{p}(D) = \mathrm{r}_{p}(\mathrm{Cl}_{D})\). Class groups are among the most fundamental objects in algebraic number theory, yet many basic questions about their structure remain open. A central problem is to understand, for an odd prime \(p\) and a non-negative integer\(r\), how often the equality \(r_{p}(D) = r\) occurs as \(D\) runs over fundamental discriminants.

In 1984, Cohen and Lenstra \cite{cl} proposed a celebrated heuristic that predicts a precise answer. Their conjecture states that for \(D > 0\) the probability that \(r_{p}(D) = r\) equals
\[
p^{-r^{2}-r}\prod_{j=r+2}^{\infty}\bigl(1-p^{-j}\bigr)\prod_{j=1}^{r}\bigl(1-p^{-j}\bigr)^{-1},
\]
while for \(D < 0\) it is
\[
p^{-r^{2}}\prod_{j=r+1}^{\infty}\bigl(1-p^{-j}\bigr)\prod_{j=1}^{r}\bigl(1-p^{-j}\bigr)^{-1}.
\]
A concise exposition of the Cohen--Lenstra heuristic for $r_p(D)=0$ can be found in Chapter~5, Section~10 of \cite{co}. The conjecture is still open for every odd prime \(p\) and every non-negative integer\(r\).

Research on the special case of the \(3\)-rank began even before the heuristic was formulated. In 1971, Davenport and Heilbronn \cite{dh}  obtained a landmark asymptotic formula for the number of \(3\)-torsion elements in \(\mathrm{Cl}_{D}\).  Hartung \cite{ha} proved the existence of infinitely many imaginary quadratic fields with \(r_{3}(D) = 0\) in 1974.  

The theorem of Davenport and Heilbronn stimulated a rich line of further investigations. For example, by refining their method, Nakagawa and Horie \cite{nh}   extended Hartung’s result to discriminants subject to arbitrary congruence conditions. Later work \cite{bst, btt, tt} provided the secondary main term and sharpened the error term in the Davenport--Heilbronn asymptotic.

Building on the result of Nakagawa and Horie as a key tool, a considerable body of literature has since been devoted to the simultaneous vanishing of the \(3\)-rank for infinite families of quadratic fields. Byeon \cite{by} initiated this direction by showing that for any fixed square-free integer \(m\) there exist infinitely many fundamental discriminants \(D > 0\) such that \(r_{3}(D)=r_{3}(mD)=0\). Subsequently, Ito \cite{it} proved that for any triple of square-free positive integer\(m_{1}, m_{2}, m_{3}\) one can find infinitely many \(D > 0\) with positive lower density satisfying \(\gcd(m_{1}m_{2}m_{3}, D) = 1\) and \(r_{3}(m_{1}D) = r_{3}(m_{2}D) = r_{3}(m_{3}D) = 0\). In a different vein, Chattopadhyay and Saikia \cite{cs} showed that for any integer \(t\) there exists a set of fundamental discriminants \(D > 0\) of positive density for which \(r_{3}(D)=r_{3}(D + 4t)=0\).

We observe that existing work mainly addresses two separate types of quadratic field families: those of the form \(\mathbb{Q}(\sqrt{mD})\) examined by Byeon and Ito, and translated families of the form \(\mathbb{Q}(\sqrt{D + t})\) examined by Chattopadhyay and Saikia. The more general combined case \(\mathbb{Q}(\sqrt{mD + t})\), and more broadly families given by \(\mathbb{Q}(\sqrt{f(D)})\) for a polynomial \(f\), has received little attention. This paper is devoted to studying the \(3\)-rank problem for families of quadratic fields of the form \(\mathbb{Q}(\sqrt{f(D)})\), where \(f\) is a general polynomial. To state our results precisely, we first introduce the following definition.

A pair of positive integers $(A, B)$ is \emph{good} if the following conditions hold.
\begin{enumerate}
\item[(i)] For every odd prime $p$ dividing $\operatorname{gcd}(A, B)$, we have $p^{2} \mid B$ but $p^{2} \nmid A$;
\item[(ii)] When $2 \mid B$, either
\begin{itemize}
\item $A \equiv 1\pmod 4$ and $B \equiv 0\pmod 4$, or
\item $A \equiv 8,12\pmod {16}$ and $B \equiv 0\pmod {16}$.
\end{itemize}
\end{enumerate}

\begin{theorem}\label{th1}
Let $n\ge1$,   $r\ge0$ and $s\ge0$ be integers satisfying \[4+r+3 s\le3^{n+1}\text{ and }r+s\ge1.\] Suppose that we have $r+s$ linear polynomials
$$f_{i}(x)=m_{i} x+n_{i},1 \le i \le r+s,$$
where the coefficients $m_i, n_i $ are positive integers. If all pairs $\left( n_{i} , m_{i}\right)$ are good for $ 1\le i\le r+s$, then there exist infinitely many fundamental discriminants $D>0$  such that the following hold 
 \begin{enumerate}
 \item $r_3(-f_i(D))<n   \text{ for all }  1\le i\le s,   \mbox{if }  s\ge1, r=0 ;$
 \item $r_3(f_i(D))<n \text{ for all } 1\le i\le r,   \mbox{if }  r\ge1, s=0;$ 
 \item $r_3(f_i(D))<n \text{ and } r_3( -f_{r+j}(D)) <n   \text{ for any }  1\le i\le r, 1\le j\le s, \mbox{if } r\ge1, s\ge1.$
 \end{enumerate}

\end{theorem}

If we take $r=0, s=3^{n}-2$ and $f_{i}(x)=x+i, 1 \le i \le s$ or $r=3^{n+1}-4, s=0$ and $f_{i}(x)=x+i, 1 \le i \le r$ in Theorem \ref{th1}, then we obtain the following corollary.

\begin{corollary}
Let $n$ be a positive integer. Then we have
\begin{enumerate}
\item There exist infinitely many fundamental discriminants $D>0$ such that for all $1 \le i \le 3^{n}-2$,
\[
r_3(-D-i)<n.
\]
\item There exist infinitely many fundamental discriminants $D>0$ such that for all $1 \le i \le 3^{n+1}-4$,
\[
r_3(D+i)<n.
\]
\end{enumerate}
\end{corollary}
A natural question is whether the linear polynomials in Theorem \ref{th1} can be replaced by polynomials of degree greater than one. For example, given positive integers $k$ and $n$, one may ask if there are infinitely many fundamental discriminants $D>0$  such that for all $1 \le i \le k$,
\[r_3(D^2+i)<n?\]
The method used to prove Theorem \ref{th1} is insufficient to address this question. However, by applying Bergelson and Leibman's theorem on polynomial extensions of van der Waerden's and Szemerédi's theorems \cite{bl}, we can establish the following result:
\begin{theorem}\label{density}
Let $k, n$ be two positive integers and $g_1(x), g_2(x), \ldots,g_k(x)$ be polynomials that take integer values at the integers and have no constant terms. Then there exist positive integers $a, d$ such that
\[r_3( a+g_i(d) )<n
\]
for all $1\le  i\le k$.
\end{theorem}

\section{Proof of Theorem \ref{density}}

To prove Theorem \ref{density}, we begin with some preliminaries. It is well known that an integer $D$ is a fundamental discriminant  of a quadratic field if either $D$ is a square-free integer with $D \equiv 1 \pmod 4 $ or $D=4m$ for some square-free integer $m \equiv 2, 3 \pmod 4 $. Throughout we denote by $D$ a fundamental discriminant.

For any positive real number $X  $,  we denote by $S_{1}(X)$  the set of positive  fundamental discriminants $0<D<X$, and by  $S_{-1}(X)$ the set of negative  fundamental discriminants  $-X<D<0$. Let $m\ge 1$ and $N \ge1$ be integers. For $\lambda=\pm1$ we define
\[
S_\lambda(X, m, N):=\left\{D: D\in S_\lambda(X),   D \equiv m\pmod N\right\}.
\]

The following theorem is due to Nakagawa and Horie \cite{nh}, which is a refinement of   Davenport-Heilbronn theorem  \cite{dh}.
\begin{theorem}[Nakagawa and Horie]\label{nh}
Suppose that the pair $(m,N)\in\mathbb{N}^2$ is good. Then we have

    \[
    \sum_{D \in S_{\lambda}(X, m, N)} 3^{r_{3}(D)} \sim c(\lambda)  \sum_{D \in S_{\lambda}(X, m, N)}1,
    \]
    where $c(\lambda)=\frac{4}{3}$ if $\lambda=1$, and $2$ if $\lambda=-1$. Moreover,
    \[\sum_{D \in S_{\lambda}(X, m, N)}1\sim \frac{3X}{\pi^2 \varphi(N)}\prod_{p|N}\frac{q}{1+p},\]
where $\varphi(\cdot)$ is the Euler's totient function, $q=4$ if $p=2$, and $q=p$ otherwise.

\end{theorem}

\begin{remark}If   $(6m,N)=1$, then the second term and the error term in the Nakagawa-Horie theorem were given by Taniguchi and Thorne \cite[Theorem 1.6]{tt}, that is,
 \[\sum_{\substack{0\le   D\le X\\D\equiv m\pmod N }} 3^{r_{3}(D)} =\frac{4X}{\pi^2 N}\prod_{p|N}\frac{1}{1-p^{-2}}   +cX^\frac{5}{6}+O(X^{18/23}N^{20/23}),
\]
where $c$ is a constant depending on $N$. 
It seems that this theorem also holds for any good pair $(m,N)$.
\end{remark}

\begin{theorem}\label{be}
Suppose that the pair $(m,N)\in\mathbb{N}^2$ is good. Then
 for any $n \in \mathbb{N}$, we have
\begin{equation}\label{1}
    \liminf_{X \rightarrow \infty} \frac{\sharp\left\{D \in S_{\lambda}(X, m, N) \mid r_{3}(D) <n \right\}}{\sharp S_{\lambda}(X, m, N)} \ge\frac{3^n-c(\lambda) }{3^{n}-1},
\end{equation}
where $c(\lambda)=\frac{4}{3}$ if $\lambda=1$, and $2$ if $\lambda=-1$.
\end{theorem}

\begin{proof}
We observe that
\[
\begin{aligned}
 \sum_{D \in S_{\lambda}(X, m, N)} 3^{r_{3}(D)}
& =\sum_{\substack{D \in S_{\lambda}(X, m, N) \\
r_{3}(D)<n}} 3^{r_{3}(D)}+\sum_{\substack{D \in S_{\lambda}(X, m, N) \\
r_{3}(D) \geq n}} 3^{r_{3}(D)} \\
& \ge \sum_{\substack{D \in S_{\lambda}(X, m, N) \\
r_{3}(D)<n}} 1+3^n\sum_{\substack{D \in S_{\lambda}(X, m, N) \\
r_{3}(D) \geq n}} 1\\
&=\sum_{\substack{D \in S_{\lambda}(X, m, N) \\
r_{3}(D)<n}} 1+3^n\left(\sum_{\substack{D \in S_{\lambda}(X, m, N) }} 1-\sum_{\substack{D \in S_{\lambda}(X, m, N) \\
r_{3}(D) < n}} 1\right)\\
&=3^n\sum_{\substack{D \in S_{\lambda}(X, m, N) }} 1+(1-3^n)\sum_{\substack{D \in S_{\lambda}(X, m, N) \\
r_{3}(D) < n}} 1.
\end{aligned}
\]
Thus
\[ \sum_{\substack{D \in S_{\lambda}(X, m, N) \\
r_{3}(D) < n}} 1
\ge \frac{3^n}{3^n-1}\sum_{\substack{D \in S_{\lambda}(X, m, N) }} 1-\frac{1}{3^n-1}\sum_{D \in S_{\lambda}(X, m, N)} 3^{r_{3}(D)}.\]
Dividing $\sum_{\substack{D \in S_{\lambda}(X, m, N) }}1$ and applying Theorem \ref{nh}, we complete the proof.

\end{proof}
\begin{proof}[Proof of Theorem \ref{density}]
Now we applying a theorem of Bergelson and Leibman \cite[Theorem B]{bl} to prove Theorem \ref{density}.
 Bergelson-Leibman  theorem states that
if $g_1(x), g_2(x), \ldots,g_k(x)$ are polynomials with rational coefficients   taking integer values at the integers and have no constant terms, then for every subset $A$ of $\{1, 2, \ldots ,n\}$ has positive upper Banach density, that is, \[\lim_{N-M\rightarrow\infty}\sup_M\frac{ |A\cap\{M+1,\ldots,N\}|}{N-M}>0,\]  there exist positive integers $a, d$ such that $a+g_1(d), a+g_2(d), \ldots, a+g_k(d)$ all belong to $A$.  Since a set has a positive the lower limit implies that it has a positive upper Banach density, Theorem \ref{density} follows immediately from Theorem \ref{be}. In fact, a quantitative bounds of Bergelson-Leibman  theorem have been obtained recently in \cite{sw}, which shows that for a positive constant $c$, any set $A \subseteq\{1, 2, \ldots ,n\}$  of density at least $n/(\log n)^c$  contains a nontrivial polynomial progression of the form $a+g_1(d), a+g_2(d),\ldots,a+g_k(d)$.
 \end{proof}

\section{Proof of Theorem \ref{th1}}
In this section we will prove  Theorem \ref{th1}. We begin by introducing some notations.
For each $i$ with $1 \le i \le r+s$, let us denote  by $u_{i}$   the highest power of $2$ dividing $m_i+n_i$, that is, $2^{u_i}||m_{i}+n_{i}$.  Put \[u=\max _{1 \le i \le r+s}\left\{u_{i}\right\}\] and denote \[B=2^{u+2}.\] It follows that for any fundamental discriminant  $D \equiv 1\pmod B$,
we have
\[f_i(D)=m_iD+n_i\equiv m_i+n_i\pmod B,\]
which implies that
\begin{equation}\label{odd}\frac{f_i(D)}{2^{u_i}}\equiv1\pmod2.\end{equation}
For convenience we define
\[
\widetilde{f}_{i}(x):=\frac{f_{i}(x)}{2^{u_{i}}}=\frac{m_{i} x+n_{i}}{2^{u_{i}}},\qquad i=1, 2, \cdots, r+s.
\]
Therefore by (\ref{odd}) we see that $\widetilde{f}_{i}(D)\in \mathbb{Z}$ and is odd for $D\equiv1\pmod B$. Let us  define $l_i$ for $1\le i\le r+s$ as

\[
l_{i}:= \begin{cases}1, & \text { if } \widetilde{f}_{i}(1) \equiv 1 \pmod 4 , \\ 4, & \text { if } \widetilde{f}_{i}(1) \equiv 3  \pmod 4.\end{cases}
\]
Note that if  $\widetilde{f}_{i}(D)$ is square-free, then $l_{i} \widetilde{f}_{i}(D)$ is a fundamental discriminant.

\begin{lemma}\label{good}
For $1\le i \le r+s$, the pair $( l_{i} \widetilde{f}_{i}(1), 4l_{i} m_{i})$ is good.
\end{lemma}

\begin{proof}
Let $p$ be an odd prime satisfying $p\mid\operatorname{gcd}(l_{i} \widetilde{f}_{i}(1), 4 l_{i} m_{i})$. Since $p\mid m_i$, $p\mid l_{i}\widetilde{f}_{i}(1)=\frac{l_i(m_i+n_i)}{2^{u_i}}$, we have $p\mid n_i$. Because $(n_i,m_i)$ is good, we find that $p^2|m_i$ and $p^2\nmid n_i$. This implies  that $p^2\mid 4l_im_i$ and $p^2\nmid l_{i}\widetilde{f}_{i}(1)$.

Now we verify that the pair $(l_{i} \widetilde{f}_{i}(1), 4 l_{i} m_{i})$ satisfies (ii) of the definition of ``good".  Since $\widetilde{f}_{i}(1)$ is odd, we have two cases to be discussed according to $\widetilde{f}_{i}(1) \equiv 1 \text{ or }3\pmod4$. If $\widetilde{f}_{i}(1) \equiv 1\pmod 4 $, then $l_{i}=1$ by definition. In this case, we have $l_{i} \widetilde{f}_{i}(1) \equiv 1\pmod 4$ and $4 l_{i} m_{i} \equiv 0\pmod 4$. If $\widetilde{f}_{i}(1) \equiv 3\pmod 4$, then $l_{i}=4$. Thus $l_{i} \widetilde{f}_{i}(1) \equiv 12\pmod{ 16} $ and $4 l_{i} m_{i} \equiv 0\pmod{ 16}$. Thus we conclude that the pair $(l_{i} \widetilde{f}_{i}(1), 4 l_{i} m_{i})$ is good.

\end{proof}

Let $\mu(n)$ denote  the usual Möbius function.
For $1 \le i \le r+s$, we define
\[
T_{i}(X):=\left\{0<D< X: D\equiv 1\pmod{2^{u_{i}+2}}, \mu(D)\mu\left(\widetilde{f}_{i}(D)\right) \neq 0\right\}
\]
and
\[
T(X):=\left\{0<D< X: D\equiv 1\pmod{ {B}},  \mu(D)\prod_{i=1}^{r+s}\mu\left(\widetilde{f}_{i}(D)\right) \neq 0\right\}.
\]

Now we are on a position to prove Theorem \ref{th1}.

\begin{proof}[Proof of Theorem \ref{th1}]
Let us denote
\[A_i=\left\{D \in T(X): \operatorname{rk}_{3}\left({\widetilde{f}_{i}(D)}\right) \geqslant n\right\}, i=1, 2, \ldots, r+s.\]
By  the inclusion-exclusion principle we have
\[
\begin{aligned}
&\sharp\left\{D \in T(X): \max _{1 \le i \le r+s} \operatorname{rk}_{3}\left({\widetilde{f}_{i}(D)}\right)<n\right\} \\
& =|\bigcap_{i=1}^{r+s}\overline{A_i}| \\
&\ge\sharp T(X) - \sum_{i=1}^{r+s}|A_i|.
\end{aligned}
\]
Since $T(X) \subseteq T_{i}(X)$,  it follows that
\[A_i\subseteq  \left\{D \in T_i(X): \operatorname{rk}_{3}\left({\widetilde{f}_{i}(D)}\right)\ge n\right\}.\]
Thus
\[
\sharp\left\{D \in T(X): \max _{1 \le i \le r+s } \operatorname{rk}_{3}\left(\mathrm{Cl}_{\widetilde{f}_{i}(D)}\right)<n\right\}\]\[\ge\sharp T(X) - \sum_{i=1}^{r+s}
\sharp\left\{D \in T_i(X): \operatorname{rk}_{3}\left({\widetilde{f}_{i}(D)}\right)\ge n\right\}.
\]
Dividing both sides by $\sharp T(X)$ gives
\begin{align}
&\quad \frac{\sharp\left\{D \in T(X): \max _{1 \le i \le r+s} \operatorname{rk}_{3}\left({\widetilde{f}_{i}(D)}\right)<n\right\}}{\sharp T(X)}\notag \\
& \geqslant 1-\sum_{i=1}^{r+s} \frac{\sharp\left\{D \in T_i(X): \operatorname{rk}_{3}\left({\widetilde{f}_{i}(D)}\right) \geqslant n\right\}}{\sharp T(X)}\notag\\
&= 1-\sum_{i=1}^{r+s} \left(1-\frac{\sharp\left\{D \in T_i(X): \operatorname{rk}_{3}\left({\widetilde{f}_{i}(D)}\right) <n\right\}}{\sharp T(X)}\right).
\label{bound}
\end{align}
We observe that each linear polynomial $l_{i} \widetilde{f}_{i}(x)$ gives  a bijection
\[
\begin{aligned}
l_{i} \widetilde{f}_{i}: T_{i}(X) & \longrightarrow   S_{\lambda_i}\left( l_{i}\widetilde{f}_{i}(X) , l_{i} \widetilde{f}_{i}(1), 4 l_{i} m_{i}\right), \\
D & \longmapsto  \lambda_i l_{i} \widetilde{f}_{i}(D),
\end{aligned}
\]
where
\[\lambda_i=\left\{
              \begin{array}{ll}
              1, &\text{ if }s=0 \text{ and } 1\le i\le r; \\
                -1, & \text{ if }r=0 \text{ and }   1\le i\le  s; \\
                1, &\text{ if }r\ge1,s\ge1 \text{ and }  1\le i\le r; \\
                -1, &\text{ if }r\ge1,s\ge1 \text{ and } r+1\le i\le r+s.
              \end{array}
            \right.
\]
It follows immediately that

\[
\sharp T_{i}(X)=\sharp S_{\lambda_i}\left( l_{i} \widetilde{f}_{i}(X),  l_{i} \widetilde{f}_{i}(1), 4 l_{i} m_{i}\right).
\]

Lemma \ref{good} tell us that the pair $( l_{i} \widetilde{f}_{i}(1), 4  l_{i} m_{i})$ is good. Since   $\mathrm{Cl}_{l_i\widetilde{f}_{i} (D)}=\mathrm{Cl}_{\widetilde{f}_{i}(D)}$, we apply Theorem \ref{be} for (\ref{bound}) and find that

\begin{align*}
&\quad \liminf_{X\rightarrow\infty} \frac{\sharp\left\{D \in T_{i}(X): \operatorname{rk}_{3}\left({\widetilde{f}_{i}(D)}\right)< n\right\}}{\sharp T(X)}\notag \\
& \ge\liminf_{X\rightarrow\infty} \frac{\sharp\left\{D \in T_{i}(X): \operatorname{rk}_{3}\left({\widetilde{f}_{i}(D)}\right)< n\right\}}{\sharp T_{i}(X)}\notag \\
&=  \liminf_{X\rightarrow\infty} \frac{\sharp\left\{D \in S_{\lambda_i}\left( l_{i}   \widetilde{f}_{i}(X),  l_{i} \widetilde{f}_{i}(1), 4  l_{i} m_{i}\right): \operatorname{rk}_{3}\left({D}\right)< n\right\}}{\sharp S_{\lambda_i}\left( l_{i}  \widetilde{f}_{i}(X),  l_{i} \widetilde{f}_{i}(1), 4  l_{i} m_{i}\right)}\notag\\
&\ge\frac{3^n-c(\lambda_i) }{3^n-1}.
\end{align*}
Inserting this inequality  into (\ref{bound}), we obtain

\begin{align}
&\liminf_{X \rightarrow \infty} \frac{\sharp\left\{D \in T(X): \max _{1 \le i \le r+s} \operatorname{rk}_{3}\left({\widetilde{f}_{i}(D)}\right)<n\right\}}{\sharp T(X)}\notag\\
&\geqslant  1-\sum_{i=1}^{r+s} \frac{c\left(\lambda_i\right)-1}{3^{n}-1}\notag\\
&=1-\frac{\frac{1}{3}r}{3^{n }-1}-\frac{s}{3^{n}-1}\notag\\
&\ge \frac{1}{3^{n+1}-3}. \label{bound1}
\end{align}

To finalize the proof of the theorem, it is sufficient to show that the cardinality of $T(X)$ has a positive density as $X$ tends to infinity.  For this purpose, we invoke a theorem established by  Tsang \cite{ts}.

\begin{theorem}[Tsang]\label{ts}
For positive integer $t$, let $f_{i}(x)$ for $1 \le i \le t$ be irreducible integral polynomials with degree not exceeding $l$, where $(l \geqslant 3)$. Let

\begin{equation}\label{omega}
\Omega(l):=\prod_{p}\left(1-\omega(p) p^{1-l}\right),
\end{equation}
where for each prime $p, \omega(p)$ denote the number of solutions for  $f_{i}(x)\equiv0\pmod{p^{l-1}}$. If $N(X)$ is the number of integers $n \le X$ such that the values $f_{i}(n)$ are all $(l-1)$-th power free, then we have

\[
N(X)=\Omega(l) X+O\left(\frac{X}{(\log X)^{ 1-\frac{2}{l+1}}}\right).
\]
The $O$-constant here depends on $t$ as well as on the coefficients of the given polynomials.
\end{theorem}

Note that
\[  T(X)= \left\{n \le \frac{X-1}{B}:  \mu(Bx+1)\prod_{i=1}^{r+s}\mu(\widetilde{f}_i(Bx+1)) \neq 0\right\}.
\]
Applying Theorem \ref{ts} with $l=3$ for irreducible polynomials $Bx+1, \widetilde{f}_i(Bx+1), i=1,2,\cdots,r+s$, we obtain the asymptotic formula

\begin{equation}\label{tx}
\sharp T(X)=\frac{\Omega(3)}{B} X+O\left(\frac{X}{\sqrt{\log X}}\right).
\end{equation}
We show that $\Omega(3)\neq 0$. To see this, let us first consider the case where
  $p=2$. We have $Bx+1\equiv1\pmod4$ and\[\widetilde{f}_i(Bx+1)\equiv\widetilde{f}_i(1)\pmod 4.\]
Since $\widetilde{f}_i(1)$ is odd for all $1\le i\le r+s$, we conclude that $w(2)=0$. Suppose that $p$ is an odd prime. We have three cases to solve the congruence equation
\[\widetilde{f}_i(Bx+1)=\frac{  m_{i} Bx+m_i+n_{i} }{4^{u_{i}}}\equiv0\pmod{p^2}.\]
Let $N_i$ be the number of solutions of the equation above. It is clear that $N_i=1$  if $p^2\nmid m_1$ or $p||m_i, p\mid n_i$. If $p||m_i, p\nmid n_i$, then  $N_i=0$. If $p^2|m_i^2, p\nmid n_i$ then $N_i=0$.  If $p^2|m_i^2, p\mid n_i$ then the condition $(m_i,n_i)$ is good implies that $p^2\nmid n_i$, so $N_i=0$. Thus we find that $w(p)\le 1$ for all $p\ge3$. Therefore we conclude that $\Omega(3)\neq 0$ and the cardinality of $ T(X)$ has a positive density.

Now combining  (\ref{bound1}) and (\ref{tx}), we find that
\[
\liminf _{X \rightarrow \infty} \frac{\sharp\left\{D \in T(X): \max _{1 \le i \le r+s} \operatorname{rk}_{3}\left({f_{i}(D)}\right)<n\right\}}{X} \geqslant \frac{\Omega(3)}{\left(3^{n+1}-3\right) B}>0.
\]
This implies that there are infinitely many fundamental discriminants $D$  satisfying $\mathrm{rk}_{3}\left({f_{i}(D)}\right)<n$ simultaneously for all $1 \le i \le r+s$. This completes the proof of Theorem \ref{th1}.
\end{proof}

\section*{Acknowledgments}

This work is supported by the Natural Science Foundation of China (11771121).

\end{document}